\documentclass[10pt]{amsart}
\usepackage[all]{xy}
\usepackage{amsmath,amsfonts,amssymb,amsthm,epsfig,amscd,psfrag,framed,comment}
\usepackage{nicefrac,xspace,tikz}
\usetikzlibrary{arrows}

\usepackage{hyperref}

\newtheorem{Theorem}{Theorem}[section]
\newtheorem{Proposition}[Theorem]{Proposition}
\newtheorem{Lemma}[Theorem]{Lemma}

\newtheorem{Corollary}[Theorem]{Corollary}

\theoremstyle{definition}
\newtheorem{Remark}[Theorem]{Remark}
\newtheorem{Warning}[Theorem]{Warning}

\def\sC{{\mathcal{C}}}
\def\sHom{{\mathcal{H}}om}
\def\sM{{\mathcal{M}}}
\def\sN{{\mathcal{N}}}
\def\Tang{{\mathcal{CT}}_m}

\def\K{\mathcal{V}ar}

\def\1{{\mathbf{1}}}

\def\k{{\mathbb C}}

\def\sl{\mathfrak{sl}}

\def\g{\mathfrak{g}}

\newcommand{\C}{\mathbb{C}}

\newcommand{\Cone}{\operatorname{Cone}}

\def\O{{\mathcal O}}

\def\sA{{\mathcal{A}}}
\def\sE{{\mathcal{E}}}

\def\sF{{\mathcal{F}}}
\def\sL{{\mathcal{L}}}

\def\sH{{\mathcal{H}}}
\def\sT{{\mathcal{T}}}
\def\sP{{\mathcal{P}}}
\def\sQ{{\mathcal{Q}}}

\def\uk{{\underline{k}}}
\def\uw{{\underline{w}}}
\def\uo{{\underline{0}}}

\def\w{{\underline{w}}}

\def\tY{\tilde{Y}}

\newcommand{\G}{\mathbb{G}}

\newcommand{\T}{\mathsf{T}}

\newcommand{\Hom}{\operatorname{Hom}}
\newcommand{\End}{\operatorname{End}}

\newcommand{\spn}{\operatorname{span}}

\begin{document}

\title[Knot homology IV]{Knot homology via derived categories of coherent sheaves IV, coloured links}

\author{Sabin Cautis}
\email{cautis@math.ubc.ca}
\address{Department of Mathematics \\ University of British Columbia \\ Vancouver BC, Canada}

\author{Joel Kamnitzer}
\email{jkamnitz@math.utoronto.ca}
\address{Department of Mathematics \\ University of Toronto \\ Toronto ON, Canada}

\begin{abstract}
We define a deformation of our earlier link homologies for fundamental representations of $ \sl_m $. The deformed homology of a link is isomorphic to the deformed homology of the disjoint union of its components. Moreover, there exists a spectral sequence starting with the old homology and converging to this deformed homology.
\end{abstract}

\maketitle
\tableofcontents

\section{Introduction}

\subsection{Geometric construction of Khovanov homology}

Khovanov homology is an invariant of links, taking values in bigraded vector spaces.  Khovanov \cite{K} originally defined this theory in an algebraic/combinatorial fashion.  In \cite{CK1} we gave a construction of Khovanov homology in the context of geometric representation theory.  Our construction used derived categories of coherent sheaves on certain flag-like varieties constructed using the affine Grassmannian of $ PGL_2$.

Khovanov homology categorifies the Jones polynomial, which is the Reshetikhin-Turaev invariant associated to the Lie algebra $ \sl_2$ and its standard representation.  One of our main motivations in \cite{CK1} was to find a setting for Khovanov homology would generalize to other semisimple Lie algebras. This was partially accomplished in \cite{CK2, C2} where we defined knot homology theories categorifying the Reshetikhin-Turaev invariants for $ \sl_m $ representations. For the case of the standard representation our construction was more or less straightforward, whereas for other fundamental representations, our construction used a categorified form of skew-Howe duality and the theory of categorical $ \sl_n$ actions \cite{CKL1}.

\subsection{Batson-Seed coloured link homology}
Recently, Batson-Seed \cite{BS} defined a Khovanov-like homology theory for coloured links, where each component of the link is coloured\footnote{This is not to be confused with the colouring of links by representations of $\sl_m$. In this paper, we will speak about links whose components are labeled by (fundamental) representations and coloured by complex numbers.}  with a complex number.  When all the components have the same colour, their theory is the same as Khovanov homology.  On the other hand, when all the components have different colours, their theory gives the tensor product of the  Khovanov homologies of the components.  Moreover, they constructed a spectral sequence from the first case to the second case.

\subsection{Geometric coloured link homology}
Batson asked us if the Batson-Seed theory had a natural realization in our geometric framework and moreover if it was possible to extend the construction from the case of $ \sl_2 $ to $ \sl_m$. In this paper we answer these questions in the affirmative.

More precisely, we do the following.
\begin{enumerate}
\item We construct a kernel-valued invariant of coloured, labeled tangles (Theorem \ref{thm:main1}).  This invariant recovers our previous construction \cite{CK1,CK2,C2} when all colours are equal.
\item Given a link $ K$ with components $ K_1, \dots, K_r $, we construct a link invariant $ \Psi(K)_{\C^r} $ taking values in the derived category of graded modules over a polynomial ring in $ r $ variables (Proposition \ref{prop:homology}). From this invariant, we obtain a spectral sequence from the doubly-graded vector space $ \Psi(K) $ to the singly-graded vector space $ \Psi(K_1) \otimes \cdots \otimes \Psi(K_r)$ (Theorem \ref{thm:main2}).
\end{enumerate}

This provides us with a coloured link homology theory and a spectral sequence as in Batson-Seed. Moreover, as in their work, when all components have different colours our homology gives the tensor product of the homology associated to the components. Note however, that our construction works for any $ \sl_m $ and all fundamental representations.

\subsection{Deformation of varieties and deformation of kernels}
Our constructions \cite{CK1, CK2, C2} made use of iterated convolutions of spherical Schubert varieties in the affine Grassmannian.  More precisely, we considered the spaces of flags of lattices
$$
Y(\uk) := \{ L_0 = \C[z]^m \subset L_1 \subset  \dots \subset L_n \subset \C[z,z^{-1}]^m : z L_i \subset L_{i-1}, \dim(L_i/L_{i-1}) = k_i \}
$$
Using the Beilinson-Drinfeld Grassmannian, it is easy to define a deformation $ Y(\uk)_{\C^n} $ of this variety.

Our constructions \cite{CK1, CK2, C2} were based on assigning certain Fourier-Mukai kernels to caps, cups, and crossings. In this paper we prove that these kernels deform (section \ref{sec:kernelfamily}). This is easy to do for caps and cups and also for crossings involving the standard representation, because in these cases the kernels are (roughly) the structure sheaves of certain correspondences which deform.

However, when a crossing involves other fundamental representations, the resulting kernel is constructed as a Rickard complex (a.k.a. Chuang-Rouquier complex) using the theory of categorical $\sl_n$-actions. The terms in this complex do not deform. Thus, it is quite surprising and non-trivial that the total complex does deform. To prove this we identify the total complex with the push-forward (from an open subset) of a line bundle, generalizing the main result of \cite{C1} (see Proposition \ref{prop:equivalence}). This generalization of \cite{C1} is the other main result in this paper. In particular, Corollary \ref{cor:affinebraid} can be used to show that there exists an action of the affine braid group on our categories (this is not used in the current paper but is of independent interest).

\subsection*{Acknowledgements}
We thank Josh Batson for telling us about his work and for raising the questions which lead to the present paper. J.K. was supported by an NSERC discovery grant, a Sloan fellowship, and a Simons fellowship. S.C. was supported by an NSERC discovery grant and the Templeton Foundation.

\section{Geometric background}\label{sec:geometry}

\subsection{Notation}\label{sec:notation}

We will only consider schemes over $\k$. Denote by $D(Y)$ the bounded derived category of coherent sheaves on a scheme $Y$. If $Y$ also carries an action of $\k^\times$ then we denote by $D^{\k^\times}(Y)$ the derived category of $\k^\times$-equivariant sheaves.

All functors will be derived. So, for example, if $f: Y \rightarrow X$ is a morphism then $f_*: D(Y) \rightarrow D(X)$ denotes the derived pushforward.

\begin{Warning} There is one important exception to this rule, namely if $j: U \rightarrow Y$ is an open embedding then $j_*$ will always denote the plain (underived) pushforward.
\end{Warning}

Recall also the terminology of Fourier-Mukai kernels \cite{H}. Given an object $\sP \in D(X \times Y)$ (whose support is proper over $X $ and $Y$) we may define the associated transform, which is the functor
\begin{align*}
\Phi_\sP : D(X) & \rightarrow D(Y) \\
\sF & \mapsto {\pi_2}_* (\pi_1^* (\sF) \otimes \sP)
\end{align*}
where $\pi_1$ and $\pi_2$ are the natural projections from $X \times Y$. The object $\sP$ is called the kernel.

The right and left adjoints of $ \Phi_\sP $ are the transforms with respect to the kernels
$$ \sP^R := \sigma(\sP^\vee) \otimes \pi_2^* \omega_X [\dim(X)] \in D(Y \times X) \text{ and } \sP^L := \sigma(\sP^\vee) \otimes \pi_1^* \omega_Y [\dim(Y)] \in D(Y \times X),$$
where $ \sigma : D(X \times Y) \rightarrow D(Y \times X) $ denotes the equivalence coming from the canonical map $ \sigma : X \times Y \rightarrow Y \times X $.  If $\sP$ induces an equivalence $\Phi_{\sP}$ then its inverse is induced by $\sP^L \cong \sP^R$.

We can express composition of transforms in terms of their kernels. If $ X, Y, Z $ are varieties and $\Phi_\sP : D(X) \rightarrow D(Y),  \Phi_\sQ : D(Y) \rightarrow D(Z) $ are transforms, then $ \Phi_\sQ \circ \Phi_\sP $ is the transform with respect to the kernel
\begin{equation*}
\sQ * \sP := {\pi_{13}}_*(\pi^*_{12}(\sP) \otimes \pi^*_{23}(\sQ)).
\end{equation*}
The operation $*$ is associative. Moreover by \cite{H} remark 5.11, we have $ (\sQ * \sP)^R \cong \sP^R * \sQ^R $ and $ (\sQ * \sP)^L \cong \sP^L * \sQ^L$.

\subsection{The varieties} \label{sec:varieties}
Fix an integer $ m \ge 2 $. We now describe the basic varieties upon which our work is based. These varieties are indexed by sequences $\uk = (k_1, \dots, k_n) $ where each $ 1 \le k_i \le m-1 $. Each such $ k_i $ corresponds to a fundamental weight $ \omega_{k_i} $ of $\sl_m$.

Fix a vector space $ \C^m $ and consider the vector space $ \C^m \otimes \C(z) $ and let $ L_0 = \C^m \otimes \C[z]$.  We will study $\C[z]$-submodules $L $ of $ \C^m \otimes \C(z) $ which contain $L_0 $, and such that $ L/L_0 $ is finite-dimensional.  Note that if $ L $ is such a subspace, then $ z|_{L/L_0} $ is a linear operator on a finite-dimensional vector space.

For $\uw = (w_1, \dots, w_n)$ with $w_i \in \C$ we define
\begin{align*}
Y(\uk)_{\uw} := &\{ (L_1, \dots, L_n) : L_i \subset \C^m \otimes \C(z), \ L_0 \subset L_1 \subset L_2 \subset \dots \subset L_n, \\
&\dim(L_i / L_{i-1}) = k_i, \text{ and } (z - w_i I) L_i \subset L_{i-1} \}.
\end{align*}
Notice that if we vary the $w_i$ we obtain a family $ Y(\uk)_{\C^n} \rightarrow \C^n$ whose fibre at $\uw \in \C^n$ is $Y(\uk)_{\uw}$.

There is a natural map $Y(\uk)_{\uw} \rightarrow Y(\uk')_{\uw'}$ where $\uk'$ and $\uw'$ are obtained from $\uk$ and $\uw$ by forgetting $k_n$ and $w_n$. It is given by forgetting $L_n$. This map is a fibre bundle with fibre $\G(k_n,m)$. To see this, suppose that we have $ (L_1, \dots, L_{n-1}) \in Y(\uk')_{\uw'}$ and are considering possible choices of $L_n$. By definition we must have $L_{n-1} \subset L_n \subset (z - w_n I)^{-1}(L_{n-1})$. Since $ (z-w_n I)^{-1}(L_{n-1}) / L_{n-1} $ is always $m$ dimensional, this fibre is a $\G(k_n, m)$. Thus $Y(\uk)_\uw$ is an iterated fibre bundle.

The varieties $ Y(\uk)_\uw $ carry tautological quotient vector bundles $ L_i/L_j $ for any $ j < i $. We will use the usual convention that on $Y(\uk)_\uw \times Y(\uk')_{\uw'}$ we denote by $L_i/L_j$ the pullback of the bundle from the first factor and $L_i'/L_j'$ the pullback from the second.

\subsection{The subvarieties}\label{sec:subvar}

If $w_i = w_{i+1}$ and $k_i+k_{i+1}=m$, we define
$$d_i \cdot (k_1, \dots, k_n) := (k_1, \dots, k_{i-1}, k_{i+2}, \dots, k_n), \quad d_i \cdot (w_1, \dots, w_n) = (w_1, \dots w_{i-1}, w_{i+2}, \dots, w_n) $$
and we define$X(\uk)_{\uw}^i \subset Y(\uk)_{\uw}$ as follows
$$X(\uk)^i_{\uw} := \{ (L_1, \dots, L_n) \in Y(\uk)_{\uw}: (z-w_iI) L_{i+1} = L_{i-1} \}.$$
These varieties are the fibres of a family $ X(\uk)^i_{\C^{n-1}} $ defined over the base $ \C^{n-1} = \{ (w_1, \dots, w_n) : w_i = w_{i+1} \} $. There exists a natural projection $q: X(\uk)^i_{\uw} \rightarrow Y(d_i \cdot \uk)_{d_i \cdot \uw} $ given by forgetting $L_i$. Notice that since $\dim(L_{i+1}/L_{i-1}) = m$, the map $q$ is a $\G(k_i,m)$-fibre bundle.

\begin{Remark}
Note that the varieties $X(\uk)_\uw^i$ are not defined if $w_i \ne w_{i+1}$. In fact, for the purpose of defining the categorical $ \sl_n$-action from \cite{CKL1}, it is very important that $X(\uk)^i_{\uw}$ do not deform over the locus where $w_i \ne w_{i+1}$.
\end{Remark}

Now, for arbitrary $\uw$ and $ \uk $, we define
\begin{equation*}
s_i \cdot (k_1, \dots, k_n) := (k_1, \dots, k_{i-1}, k_{i+1}, k_i, k_{i+2}, \dots, k_n)
\end{equation*}
and similarly, $ s_i \cdot \uw $.  Then we define the subvariety
\begin{equation*}
 Z(\uk)^i_{\uw} := \{ (L_\cdot, L'_\cdot) : L_j = L_j' \text{ for $ j \ne i $ } \} \subset Y(\uk)_\uw \times Y(s_i \cdot \uk)_{s_i \cdot \uw}.
\end{equation*}
As before, the varieties $ Z(\uk)^i_\uw $ are the fibres of a family $ Z(\uk)^i_{\C^n} \rightarrow \C^n $.   Note that $ Z(\uk)^i_{\C^n} \subset Y(\uk)_{\C^n} \times_{\C^n}  Y(s_i \cdot \uk)_{\C^n} $ where in forming the fibre product we twist the second map by $ s_i $.

\begin{Lemma}\label{lem:graph}
If $w_i \ne w_{i+1}$ then $ Z(\uk)^i_\uw$ is the graph of an isomorphism.
\end{Lemma}
\begin{proof}
Consider a point $ L_\cdot \in Y(\uk)_\uw$ and let $ (L_i, L'_i) \in Z(\uk)^i_\uw $ be a point in the fibre of $ Z(\uk)^i_{\uw} \rightarrow Y(\uk)_\uw $ over $ L_i$.  The linear operator $z : L_{i+1}/L_{i-1} \rightarrow L_{i+1}/L_{i-1}$ is diagonalizable with two eigenvalues $w_i$ and $w_{i+1}$.  So the only choice for $L'_i \in Z(\uk)^i_{\uw}$ is to take $L'_i/L_{i-1} $ to be the $w_{i+1}$-eigenspace of $z|_{L_{i+1}/L_{i-1}}$.  Thus the fibre has one point and so the projection $ Z(\uk)^i_{\uw} \rightarrow Y(\uk)_\uw$ is an isomorphism.  Similarly, $ Z(\uk)^i_{\uw} \rightarrow Y(s_i \cdot \uk)_{s_i \cdot \uw}$ is an isomorphism.
\end{proof}

\subsection{$\C^\times$ action}\label{sec:c*actions}

Consider the action of $\C^\times$ on $\C(z)$ given by $t \cdot z = t^2 z$.  This induces an action of $ \C^\times $ on $ \C(z) \otimes \C^m $. Thus, given a subspace $ L \subset \C(z) \otimes \C^m $, we can consider $ t L $.  Note that if $ (z-wI)L \subset L' $, then $ (z-t^{-2} wI) tL \subset tL' $. Hence, the $\C^\times$ action on $ \C(z)$ induces a $\C^\times$ action on $Y(\uk)_{\C^n}$ by
$$t \cdot (L_1, \dots, L_{n}) = (t \cdot L_1, \dots, t \cdot L_{n}).$$
This action is compatible with the scaling action (with exponent -2) of $ \C^\times $ on $\C^n$. In particular, this $\C^\times$ does not act on $Y(\uk)_\uw$ for general $\uw$ but it does act on the central fibre $Y(\uk)_{(0,\dots,0)}$.

We denote by $\O_Y\{p\}$ the structure sheaf of $Y$ with non-trivial $\C^\times$ action of weight $p$. This means that on $Y(\uk)_{(0,\dots,0)}$ we have natural maps of vector bundles $z: L_i \rightarrow L_{i-1} \{2\}$.

\subsection{The Beilinson-Drinfeld Grassmannian}
The family $Y(\uk)_{\C^n} $ admits a natural interpretation using the Beilinson-Drinfeld Grasssmannian \cite{MV}.

We can describe $Y(\uk)_\uw$ using the Beilinson-Drinfeld Grassmannian of $PGL_m$ as follows.  First if $ V, V' $ are rank $m $ locally-free sheaves on $ \mathbb P^1 $, then we say that $ \phi : V \rightarrow V' $ is a Hecke modification of type $ \omega_k $ at $ w \in \mathbb P^1$, if $ \phi $ is injective and $ coker(\phi) $ is isomorphic to $ \O_w^{\oplus k} $. The following result is well-known to experts.

\begin{Proposition}
We have an isomorphism
\begin{equation}\label{eq:iso}
\begin{aligned}
Y(\uk)_{\uw} \cong \{ &(V_0, \dots, V_n), (\phi_1, \dots, \phi_n) :  \\
&\phi_i : V_{i-1} \rightarrow V_i \text{ is a Hecke modification of type $ \omega_{k_i} $ at $ w_i $ and $ V_0 = \O^m$} \}.
\end{aligned}
\end{equation}
\end{Proposition}
\begin{proof}
We will just give a sketch of the proof since this result is only used for motivation.

Let $ (V_0, \dots, V_n), (\phi_1, \dots, \phi_n) $ be a point on the right hand side.  Let $ L_i \subset \C(z) \otimes \C^m $ be the set of rational sections $ s $ of $ V_0 $ such that $ \phi_i \circ \dots \circ \phi_1(s) $ is a regular section of $ V_i $.  This defines a map from the right hand side to the left hand which is easily seen to be invertible.
\end{proof}

\begin{Remark}
The right hand side of the isomorphism in (\ref{eq:iso}) is a subvariety of the Beilinson-Drinfeld Grassmannian of $ \mathbb P^1 $. From this perspective the $\C^\times$ action discussed above comes from the action of $\C^\times$ on the curve $ \mathbb P^1 $.
\end{Remark}

\section{Kernels from tangles} \label{se:funtangle}

Consider the category $\Tang$ of oriented, coloured tangles labeled by $1,\dots,m-1$. More precisely, an object in $\Tang$ is a pair $(\uk,\uw)$ where $\uk=(k_1, \dots, k_n)$ with $k_i \in \{1, \dots, (m-1)\}$ and $\uw = (w_1, \dots, w_n) $ with $ w_i \in \k $.

An oriented tangle is a smooth embedding of $(n_1+n_2)/2$ many oriented arcs and finitely many circles into $ \mathbb{R}^2 \times [0,1] $, such that the boundary points of the arcs map bijectively onto the $n_1+n_2$ points $ (1,0,0), \dots, (n_1,0,0)$ and $(1,0,1), \dots, (n_2,0,1) $. An oriented, coloured, labeled tangle is an oriented tangle such that each component carries a label $ k \in \{1, \dots, m-1\} $ and a colour $w \in \k$.

Given an oriented, coloured, labeled tangle, we can read off the labels and colours on its top and bottom endpoints, thus giving two objects in $\Tang$.  The orientation of the tangle affects the labelling of the endpoints. If a strand labeled by $k$ is oriented upward, then the endpoint is labeled $k$; while if it is oriented downward, then it is labeled $m-k$.

Isotopy of tangles is defined in the usual way, except that strands carrying different colours $w_i \ne w_j$ are allowed to pass through each other.

The set of morphisms $\Hom((\uk_1,\uw_1),(\uk_2,\uw_2))$ in $\Tang$ consists of the isotopy classes of all tangles with these endpoints. Composition is given by concatenating the tangles.

\subsection{Generators and relations}

The morphisms in $\Tang$ are generated by the four crossings illustrated in figure \ref{fig:1}, their inverses, and by cups and caps. The strands in figure \ref{fig:1} are also allowed to carry arbitrary colours (which we omit for convenience).  We will give these generators the following names: $ t_1(\uk)^i_\uw, t_2(\uk)^i_\uw, t_3(\uk)^i_\uw, t_4(\uk)^i_\uw, c_1(\uk)^i_\uw, c_2(\uk)^i_\uw $ and we observe the following conventions. For the crossings and cap ($c_1$), the labels $ \uk $ and colours $ \uw $ refer to the labels and colours on the bottom endpoints, whereas for the cup ($c_2$), the labels and colours refer to those on the top endpoints.  We are abusing notation in (at least) two respects.   First, there really should be another cap, obtained by reversing orientation and labelling --- however, since our functor will end up assigning the same value to this other generators, we will ignore it (similar for cup).  Second, in our notation we are ignoring orientations on the other strands.

\begin{figure}[!ht]
\centering
\begin{tikzpicture}[>=stealth]
\draw [->](0,0) -- (1,1) [very thick];
\draw [->](0.4,0.6) -- (0,1) [very thick];
\draw [-](1,0) -- (0.6,0.4) [very thick];
\draw [shift={+(0,-0.2)}](0,0) node {$k_i$};
\draw [shift={+(0,-0.2)}](1,0) node {$k_{i+1}$};
\draw (0.5,-1) node {$t_1(\uk)^i_\uw$};

\draw [shift={+(2,0)}][->](0,0) -- (1,1) [very thick];
\draw [shift={+(2,0)}][-](0.4,0.6) -- (0,1) [very thick];
\draw [shift={+(2,0)}][<-](1,0) -- (0.6,0.4) [very thick];
\draw [shift={+(2,0)}][shift={+(0,-0.2)}](0,0) node {$k_i$};
\draw [shift={+(2,0)}][shift={+(0,-0.2)}](1,0) node {$k_{i+1}$};
\draw [shift={+(2,0)}](0.5,-1) node {$t_2(\uk)^i_\uw$};

\draw [shift={+(4,0)}][<-](0,0) -- (1,1) [very thick];
\draw [shift={+(4,0)}][->](0.4,0.6) -- (0,1) [very thick];
\draw [shift={+(4,0)}][-](1,0) -- (0.6,0.4) [very thick];
\draw [shift={+(4,0)}][shift={+(0,-0.2)}](0,0) node {$k_i$};
\draw [shift={+(4,0)}][shift={+(0,-0.2)}](1,0) node {$k_{i+1}$};
\draw [shift={+(4,0)}](0.5,-1) node {$t_3(\uk)^i_\uw$};

\draw [shift={+(6,0)}][<-](0,0) -- (1,1) [very thick];
\draw [shift={+(6,0)}][-](0.4,0.6) -- (0,1) [very thick];
\draw [shift={+(6,0)}][<-](1,0) -- (0.6,0.4) [very thick];
\draw [shift={+(6,0)}][shift={+(0,-0.2)}](0,0) node {$k_i$};
\draw [shift={+(6,0)}][shift={+(0,-0.2)}](1,0) node {$k_{i+1}$};
\draw [shift={+(6,0)}](0.5,-1) node {$t_4(\uk)^i_\uw$};

\draw [shift={+(8,.5)}](0,0) -- (0,-.5) [very thick];
\draw [shift={+(8,.5)}](1,0) -- (1,-.5)[-] [very thick];
\draw [shift={+(8,.5)}](0,0) arc (180:0:.5) [very thick];
\draw [shift={+(8,0)}][shift={+(0,-0.2)}](0,0) node {$k_i$};
\draw [shift={+(8,0)}][shift={+(0,-0.2)}](1,0) node {$m-k_i$};
\draw [shift={+(8,0)}](0.5,-1) node {$c_1(\uk)^i_\uw$};

\draw [shift={+(10,.2)}](0,0) -- (0,.5) [-][very thick];
\draw [shift={+(10,.2)}](1,0) -- (1,.5) [very thick];
\draw [shift={+(10,.2)}](0,0) arc (180:360:.5) [very thick];
\draw [shift={+(10,0)}][shift={+(0,.9)}](0,0) node {$k_i$};
\draw [shift={+(10,0)}][shift={+(0,.9)}](1,0) node {$m-k_i$};
\draw [shift={+(10,0)}](0.5,-1) node {$c_2(\uk)^i_\uw$};

\end{tikzpicture}
\caption{The cap and cup can have either orientation.}\label{fig:1}
\end{figure}
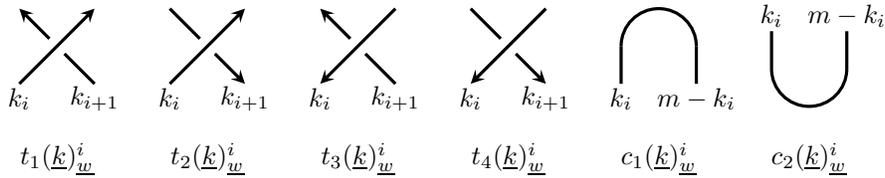

\begin{figure}[!ht]
\centering
\begin{tikzpicture}[>=stealth]
\draw [-](0,0) -- (.7,.7) [very thick];
\draw [->](0.4,0.6) -- (0,1) [very thick];
\draw [-](1,0) -- (0.6,0.4) [very thick];
\draw (0.7,0.7) arc (125:65:.5) [very thick];
\draw [->](1.2,0.75) -- (2,0) [very thick];
\draw [shift={+(0,-0.2)}](0,0) node {$k_i$};
\draw [shift={+(0,-0.2)}](1,0) node {$k_{i+1}$};
\draw [shift={+(0,-0.2)}](2,0) node {$m-k_i$};

\draw [shift={+(0,0)}](3,.5) node {$=$};

\draw [shift={+(4,0)}][-](0,0) -- (.7,.7) [very thick];
\draw [shift={+(4,0)}][-](1,0) -- (1.4,0.4) [very thick];
\draw [shift={+(4,0)}][->](1.6,0.6) -- (2,1) [very thick];
\draw [shift={+(4,0)}](0.7,0.7) arc (125:65:.5) [very thick];
\draw [shift={+(4,0)}][->](1.2,0.75) -- (2,0) [very thick];
\draw [shift={+(4,0)}][shift={+(0,-0.2)}](0,0) node {$k_i$};
\draw [shift={+(4,0)}][shift={+(0,-0.2)}](1,0) node {$k_{i+1}$};
\draw [shift={+(4,0)}][shift={+(0,-0.2)}](2,0) node {$m-k_i$};
\end{tikzpicture}
\caption{One of the fork relations.}\label{fig:2}
\end{figure}
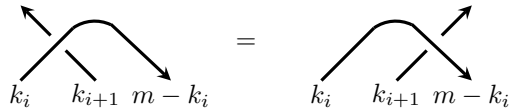

\begin{Lemma}\label{lem:relations}
Two tangle diagrams represent isotopic tangles if and only if one can be obtained from the other by applying a finite number of the following operations:
\begin{itemize}
\item a Reidemeister move of type 0,I,II or III,
\item a pitchfork move, one of which is illustrated in Figure \ref{fig:2},
\item an isotopy exchanging the order with respect to height of two caps, cups, or crossings,
\item passing two strands with different colours through each other.
\end{itemize}

More precisely, we have the following relations and those obtained from these by reflection, etc.  We have generally suppressed the colours $\uw $ and the labels $ \uk $ from the notation to make it simpler.
\begin{itemize}
\item Reidemeister (0) : $ c_1^i \circ c_2^{i+1} = id = c_1^{i+1} \circ c_2^i$,
\item Reidemeister (I) : $c_1^i \circ t_2^{i} = c_1^i$ and $c_1^i \circ t_3^i = c_1^i$,
\item Reidemeister (II) : $ t_1^i, t_2^i, t_3^i, t_4^i $ are invertible,
\item Reidemeister (III) : $ t_{l_1}^i \circ t_{l_2}^{i+1} \circ t^i_{l_3} = t_{l_3}^{i+1} \circ t_{l_2}^i \circ t_{l_1}^{i+1}$,
\item changing height isotopies, such as : $ c_1^i \circ c_1^{i+k} = c_1^{i+k-2} \circ c_1^i $ for $ k \ge 1$,
\item pitchfork move : $ c_1^{i+1} \circ t_1^i = c_1^i \circ (t_2^{i+1})^{-1}$,
\item passing two strands with different colours through each other: $ t_1(s_i \cdot \uk)^i_{s_i \cdot \uw} \circ t_1(\uk)^i_\uw = id $ if $ w_i \ne w_{i+1} $.
\end{itemize}
\end{Lemma}

\subsection{From tangles to kernels}\label{sse:functor}

We now define a functor $\Psi : \Tang \rightarrow \K$ from the category of oriented, labeled, coloured tangles to the category $\K$ whose objects are varieties and whose morphisms are isomorphism classes of kernels.

At the level of objects we map $(\uk,\uw)$ to $Y(\uk)_\uw$. At the level of morphisms we have to explain where all the generators are mapped.

\subsubsection{Caps and cups}
We begin with caps and cups.  Let $ \uk, \uw $ be such that $ k_i + k_{i+1} = m $ and $ w_i = w_{i+1} $.  We define
\begin{align*}
\Psi(c_1(\uk)^i_\uw) &= \sC_1(\uk)^i_\uw := \O_{X(\uk)^i_\uw}  \\
\Psi(c_2(\uk)^i_\uw) &= \sC_2(\uk)^i_\uw := \sigma(\O_{X(\uk)^i_\uw} \otimes \det(L_i/L_{i-1})^{m-k_i} \otimes \det(L_{i+1}/L_i)^{-k_i})
\end{align*}
Here $ X(\uk)^i_\uw $ is viewed as a subvariety of $ Y(\uk)_\uw \times Y(d_i \cdot \uk)_{d_i \cdot \uw} $, so that $ \sC_1(\uk)^i_\uw $ is a kernel on $ Y(\uk)_\uw \times Y(d_i \cdot \uk)_{d_i \cdot \uw} $ whereas $ \sC_2(\uk)^i_\uw $ is a kernel on $ Y(d_i \cdot \uk)_{d_i \cdot \uw} \times Y(\uk)_\uw $.

\subsubsection{Crossings with different labels}
To describe where to take the crossings we have two cases, depending on whether or not $w_i$ and $w_{i+1}$ are the same. If $w_i \ne w_{i+1}$ then we define
\begin{equation}\label{eq:crossing1}
\Psi(t_l(\uk)^i_\uw) = \sT_l(\uk)^i_\uw := \O_{Z(\uk)^i_\uw} \otimes \rho
\end{equation}
for $ l = 1, 2,3,4$. Here $\rho$ is the line bundle $\det(L_{i+1}/L_i)^{k_i} \otimes \det(L_{i+1}'/L_i')^{-k_{i+1}}$ on $Y(\uk)_\uw \times Y(s_i \cdot \uk)_{s_i \cdot \uw}$. In this case $Z(\uk)^i_\uw$ is the graph of an isomorphism so these kernels are invertible. So we can map the four reversed crossings to the corresponding inverses.

\subsubsection{Crossings with the same labels}
If $w_i = w_{i+1}$ then the maps are a little more difficult to describe. Without loss of generality (and to simplify notation) let us assume $w_i=w_{i+1}=0$ and that $k_i \le k_{i+1}$. We let $N := \min(k_i+k_{i+1},m)$. Then $Z(\uk)^i_\uw$ consists of $N-k_{i+1}+1$ components denoted
\begin{equation}\label{eq:components}
Z_s(\uk)^i_\uw = \{(L_\cdot,L'_\cdot) \in Z(\uk)^i_\uw: \dim (\ker z|_{L_{i+1}/L_{i-1}}) \ge k_{i+1}+s \text{ and } \dim ((L_i \cap L_i')/L_{i-1}) \ge k_i-s \}
\end{equation}
where $s = 0, \dots, N-k_{i+1}$. Since $\spn(L_i,L'_i)/L_{i-1} \subset \ker z|_{L_{i+1}/L_{i-1}}$ it follows that
$$\dim (\ker z|_{L_{i+1}/L_{i-1}}) + \dim ((L_i \cap L_i')/L_{i-1}) \ge k_i+k_{i+1}$$
in $Z(\uk)^i_\uw$. We define the open subscheme
\begin{equation}\label{eq:open}
Z^o(\uk)^i_\uw := \{(L_\cdot,L'_\cdot) \in Z(\uk)^i_\uw: \dim (\ker z|_{L_{i+1}/L_{i-1}}) + \dim ((L_i \cap L_i')/L_{i-1}) \le k_i+k_{i+1}+1 \}.
\end{equation}
and $Z^o_s(\uk)^i_\uw := Z_s(\uk)^i_\uw \cap Z^o(\uk)^i_\uw$. We denote by $j$ the open embedding of $Z^o(\uk)^i_\uw$ into $Z(\uk)^i_\uw$. The intersections $D_{s,+}^o := Z^o_s(\uk)^i_\uw \cap Z_{s+1}(\uk)^i_\uw$ and $D_{s,-}^o := Z^o_s(\uk)^i_\uw \cap Z_{s-1}(\uk)^i_\uw$ are divisors in $Z^o_s(\uk)^i_\uw$.

\begin{Lemma}\label{lem:lb}
There exists a line bundle $\sL(\uk)^i_\uw$ on $Z^o(\uk)^i_\uw$ uniquely determined by the property that when restricted to each component $Z^o_s(\uk)^i_\uw$ it is isomorphic to
$$\O_{Z^o_s(\uk)^i_\uw}([D^o_{s,+}]) \otimes \det(L_{i+1}/L_i)^{-s} \otimes \det(L_i'/L_{i-1})^s \otimes \rho$$
where, as before, $\rho = \det(L_{i+1}/L_i)^{k_i} \otimes \det(L_{i+1}'/L_i')^{-k_{i+1}}$.
\end{Lemma}
\begin{proof}
This follows from Proposition \ref{prop:equivalence}.
\end{proof}

We now define map the four crossings in Figure \ref{fig:1} to
\begin{equation}\label{eq:crossing2}
\Psi(t_l(\uk)^i_\uw) = \sT_l(\uk)^i_\uw := j_* \sL(\uk)^i_\uw  \otimes \rho
\end{equation}
for $ l = 1, 2, 3, 4$.  Proposition \ref{prop:equivalence} implies that these kernels are invertible. So we map the four reverse crossings to the corresponding inverses.

\begin{Remark}
This definition of a crossing when $w_i=w_{i+1}$ is necessarily more complicated because, as noted by Namikawa in \cite{N} (in the closely related context of cotangent bundles to Grassmannians), the sheaf $\O_{Z(\uk)^i_{\uw}}$ does {\em not} induce an equivalence of categories if $m=4, k_i=k_{i+1}=2$ and $w_i=w_{i+1}$. This means that Reidemeister II cannot hold. More generally one expects that $\O_{Z(\uk)^i_\uw}$ fails to induce an equivalence if $w_i = w_{i+1}$ and $k_i,k_{i+1} \notin \{1,m-1\}$.
\end{Remark}

Finally, to make $\Psi: \Tang \rightarrow \K$ well defined with respect to Reidemeister I we shift the image of a tangle $T$ in $\K$ by $[\sum_k d_k k(m-k)]$ where $d_k$ is the number of positive minus the number of negative crossings in $T$ involving two strands labeled by $k$ (in figure \ref{fig:1} crossings $t_2$ and $t_3$ are positive while $t_1$ and $t_4$ are negative). Without this shift Reidemeister I would state that a positive curl involving a strand labeled $k$ can be undone at the cost of shifting by $[-k(m-k)]$.

\begin{Theorem}\label{thm:main1}
The maps above describe a well defined functor $\Psi: \Tang \rightarrow \K$.
\end{Theorem}

\section{Proof of Theorem \ref{thm:main1}}\label{sec:mainthm}
We must prove that the relations from Lemma \ref{lem:relations} hold.

\subsection{Strands coloured the same}
When $ w_i = w_{i+1} $, then the kernels used to define the cups, caps and crossings come from the categorical $\sl_n$ action from \cite[Section 8]{C2} (which differs from that in \cite{CKL1} by conjugation with a line bundle, see Remark \ref{rem:kernel}).  For cups and caps, this is clear and for crossings this is proved in Proposition \ref{prop:equivalence}.

Thus the results from \cite{C2} show that when all strands have the same colour, then these kernels satisfy all the relations from Lemma \ref{lem:relations}.

\subsection{Strands coloured differently}

It remains to show these relations when at least two of the strands are coloured differently.   In this case there is no version of Reidemeister move 0 or I. Reidemeister move II is easy since $Z(\uk)_\uw^i$ is the graph of an isomorphism. Reidemeister III involves two cases depending on whether or not all three strands are coloured differently.

\subsubsection{Reidemeister III}
If all three strands are coloured differently then each crossing is given by the graph of an isomorphism so this is easy to check. We now prove the case when the three strands are coloured by $w_i=w_{i+1} \ne w_{i+2}$ (the cases $w_i=w_{i+2} \ne w_{i+1}$ and $w_i \ne w_{i+1}=w_{i+2}$ are dealt with similarly).

We need to show that
$$\sT_{l_3}(s_{i+1}s_i \cdot \uk)^i_{s_{i+1}s_i \cdot \uw} * \sT_{l_2}(s_i \cdot \uk)^{i+1}_{s_i \cdot \uw} * \sT_{l_1}(\uk)^i_{\uw} \cong \sT_{l_1}(s_is_{i+1} \cdot \uk)^{i+1}_{s_is_{i+1} \cdot \uw} * \sT_{l_2}(s_{i+1} \cdot \uk)^{i}_{s_{i+1} \cdot \uw} * \sT_{l_3}(\uk)^{i+1}_{\uw}.$$
For simplicity we omit the $\uk$,$\uw,l_i$ and assume all $l_i=1$. On the left hand side the right most $\sT^i$ is $j_* \sL(\uk)^i_\uw$ while the subsequent $\sT^i * \sT^{i+1}$ are both given (up to tensoring by line bundles) by the graphs of isomorphisms as in Lemma \ref{lem:graph}. In particular, the left hand side is supported on the variety
$$Z := \{(L_\cdot, L'_{\cdot}): L_j = L_j' \text{ for } j \ne i,i+1 \} \subset Y(\uk)_\uw \times Y(s_i s_{i+1} s_i \cdot \uk)_{s_i s_{i+1} s_i \cdot \uw}.$$
This variety is isomorpic to $Z(\uk)^i_\uw$ using the same argument as in Lemma \ref{lem:graph}. Namely, to recover $Z(\uk)^i_\uw$ from $Z$ we need to recover $L_i'$. This can be done by taking $L_i'$ inside $Z(\uk)^i_\uw$ to be the $w_i$-eigenspace of $z|_{L'_{i+1}/L_{i-1}}$ inside $Z$. One can similarly recover $Z$ from $Z(\uk)^i_\uw$.

It then follows that the left hand side is isomorphic to $j'_* \sL(\uk)^i_\uw$ where $j'$ is the embedding of $Z^o(\uk)^i_\uw$ inside $Z$ under the isomorphism $Z \cong Z(\uk)^i_\uw$ mentioned above. One can similarly show that the right hand side is induced by the same kernel $j'_* \sL(\uk)^i_\uw$.

\subsubsection{Other relations}
Next one needs to prove the fork moves (such as the one in figure \ref{fig:2}) when the two strands are coloured differently. This is straightforward to do because now the crossings now are the graph of an isomorphism.

The isotopy relations are clear because the definitions of the kernels only involve nearby strands. Finally, if $w_i \ne w_{i+1}$ then we need to show that
$$\sT_{l}(s_i \cdot \uk)^i_{s_i \cdot \uw} * \sT_{l}(\uk)^i_\uw \cong \O_\Delta.$$
This is because on the left hand side, $L_i''/L_{i-1}$ is the $w_i$-eigenspace of $L_{i+1}/L_{i-1}$ and hence equal to $L_i/L_{i-1}$. Moreover, the line bundles on the left hand side are isomorphic to
$$\left( \det(L_{i+1}/L_i)^{k_i} \otimes \det(L_{i+1}/L_i')^{-k_{i+1}} \right) \otimes \left( \det(L_{i+1}/L_i')^{k_{i+1}} \otimes \det(L_{i+1}/L_i'')^{-k_i} \right)$$
which is trivial (since $L_i = L_i''$). This shows that strands with different colours pass through each other which completes the proof of Theorem \ref{thm:main1}.

\section{Link homology theory}

\subsection{Equivariant enhancement}

Let $K$ be a link with $r$ components $K_1, \dots, K_r$ coloured with $w_1, \dots, w_r$. Theorem \ref{thm:main1} associates to $K$ a kernel $\Psi(K)_{\uw}$ belonging to $D(pt)$ ({\it i.e.} a complex of vector spaces).

Suppose that $ \uw = \uo = (0, \dots, 0)$.   In this case, all the varieties involved carry an action of $ \k^\times $ and we can work in the categories of $\k^\times$-equivariant kernels. To do this we must modify our above definitions as follows:
\begin{align*}
\Psi(c_1(\uk)^i_\uo) &= \sC_1(\uk)^i_\uo := \O_{X(\uk)^i_\uo}  \\
\Psi(c_2(\uk)^i_\uo) &= \sC_2(\uk)^i_\uo := \sigma(\O_{X(\uk)^i_\uo} \otimes \det(L_i/L_{i-1})^{m-k_i} \otimes \det(L_{i+1}/L_i)^{-k_i}) \{k_i(m-k_i)\} \\
\Psi(t_1(\uk)^i_\uo) &= \sT_1(\uk)^i_\uo := j_* \sL(\uk)^i_\uo,  \\
\Psi(t_2(\uk)^i_\uo) &= \sT_2(\uk)^i_\uo := j_* \sL(\uk)^i_\uo\{-k_i\} \\
\Psi(t_3(\uk)^i_\uo) &= \sT_3(\uk)^i_\uo := j_* \sL(\uk)^i_\uo\{-m+k_{i+1}\} \\
\Psi(t_4(\uk)^i_\uo) &= \sT_4(\uk)^i_\uo := j_* \sL(\uk)^i_\uo\{k_{i+1}-k_i\}.
\end{align*}
\begin{Remark}\label{rem:shift}
As before, to make $\Psi(K)_{\uo}$ invariant under Reidemeister I we shift the definition above by $[\sum_k d_k k(m-k)]\{- \sum_k d_k k (m-k)\}$ where $d_k$ is the number of positive minus the number of negative crossings in $K$ involving two strands labeled $k$.
\end{Remark}

Using the same arguments as before, but keeping track of the $\C^\times$ equivariance, it is straightforward to check that with these definitions give a tangle invariant (when all strands coloured 0). In fact, the somewhat strange looking grading shifts above are uniquely determined by requiring that all the fork moves and Reidemeister II moves hold. The resulting link invariant $ \Psi(K)_{\uo} $ belongs to $ D^{\C^\times}(pt) $ ({\it i.e.} a complex of graded vector spaces) and recovers our earlier doubly graded homology from \cite{CK1,CK2,C2}.

\subsection{Family of homology theories}
On the other hand, if all $w_i$ are distinct, then $\Psi(K)_{\uw} \cong \Psi(K_1)_{w_1} \otimes \cdots \Psi(K_r)_{w_r}$ because strands of different colours can pass through each other.

The doubly graded homology theory at the central fibre and the singly graded homology theory at the general fibre are related as follows.
\begin{Theorem}\label{thm:main2}
There exists a spectral sequence which starts at $\Psi(K)_{\uo}$ and converges to $\Psi(K_1)_{w_1} \otimes \dots \otimes \Psi(K_r)_{w_r}$.
\end{Theorem}

The key to proving Theorem \ref{thm:main2} is realizing that $\Psi(K)_{\uw}$ actually gives us a family of homologies. In other words, using the notation above, there exists $\Psi(K)_{\C^r} \in D^{\k^\times}(\k^r)$  which specializes to $ \Psi(K)_{\uw} $ for each $\uw \in \k^r$.

\subsection{Kernels in families} \label{sec:kernelfamily}

Recall that our varieties $Y(\uk)_\uw$ come in families $Y(\uk)_{\k^n} \rightarrow \k^n$. Moreover, one can show that all the generating kernels we defined also live in families.

More precisely, let $ \uk $ be such that $ k_i + k_{i+1} = m $.  We define
$$
\sC_1(\uk)^i_{\C^{n-1}} := \O_{X(\uk)^i_{\C^{n-1}}} \in D^{\k^\times}(Y(\uk)_{\C^{n-1}} \times_{\C^{n-2}} Y(d_i \cdot \uk)_{\C^{n-2}}).
$$
Here in the definition of the fibre product $Y(\uk)_{\C^{n-1}} \times_{\C^{n-2}} Y(\uk)_{\C^{n-2}}$, the map from the first factor is uses $ d_i$. Similarly, we define
$$ \sC_2(\uk)^i_{\C^{n-1}} := \sigma(\O_{X(\uk)^i_{\C^{n-1}}} \otimes \det(L_i/L_{i-1})^{m-k_i} \otimes \det(L_{i+1}/L_i)^{-k_i}) \{k_i(m-k_i)\}. $$

For crossings, the situation is a little bit more complicated, since we originally defined $ \sT(\uk)^i_{\uw} $ using a complicated sheaf when $ w_i = w_{i+1} $.

Let $ \uk $ be any sequence (for notational convenience suppose $k_i \le k_{i+1}$). Recall that we have the correspondence $ Z(\uk)_{\C^n}^i \subset Y(\uk)_{\C^n} \times_{\C^n} Y(s_i \cdot \uk)_{\C^n}$.  We define
$Z^o(\uk)^i_{\k^n} \subset Z(\uk)^i_{\k^n}$ to be the open locus given by the condition
$$k_i+k_{i+1}+1 \ge \dim(\ker(z-w_iI)|_{L_{i+1}/L_{i-1}}) + \dim((L_i \cap L_i')/L_{i-1}).$$
Let $ j : Z^o(\uk)^i_{\k^n} \rightarrow Z(\uk)^i_{\k^n} $ denote the inclusion.

The restriction of $Z(\uk)^i_{\C^n}$ to $w_i=w_{i+1}$ gives us a scheme with $N-k_{i+1}+1$ components where $N = \min(k_i+k_{i+1},m)$. These components, denoted $ Z_s(\uk)^i_{\C^n} $ for $ s = 0, \dots, N - k_{i+1}$, are given by
$$ w_i = w_{i+1}, \ \ \dim(\ker(z-w_iI)|_{L_{i+1}/L_{i-1}}) \ge k_{i+1}+s \ \ \text{ and } \ \ \dim((L_i\cap L'_i)/L_{i-1}) \ge k_i-s.
$$
We also let $ Z_s^o(\uk)^i_{\k^n} := Z_s(\uk)^i_{\k^n} \cap Z^o(\uk)^i_{\k^n} $. Finally, we define
$$ \sT_1(\uk)^i_{\C^n} := j_* \sL(\uk)_{\k^n}^i \in D^{\C^\times}(Y(\uk)_{\C^n} \times_{\C^n} Y(s_i \cdot \uk)_{\C^n}) $$
where  $\sL(\uk)_{\k^n}^i $ is the $\C^\times$ equivariant line bundle on $ Z^o(\uk)^i_{\C^n}$ defined by
$$\sL(\uk)_{\k^n}^i := \O_{Z^o(\uk)^i_{\k^n}} \left( \sum_{s=0}^{N-k_{i+1}} \binom{s+1}{2} [Z^o_s(\uk)_{\k^n}^i] \right) \otimes \rho \{k_i(k_{i+1}-k_i-1)\}$$
where $\rho = \det(L_{i+1}/L_i)^{k_i} \otimes \det(L_{i+1}/L_i')^{-k_{i+1}}$. Similarly, we define
\begin{align*}
\sT_2(\uk)^i_{\C^n} &:= j_* \sL(\uk)_{\k^n}^i\{-k_i\} \\
\sT_3(\uk)^i_{\C^n} &:= j_* \sL(\uk)_{\k^n}^i\{-m+k_{i+1}\} \\
\sT_4(\uk)^i_{\C^n} &:= j_* \sL(\uk)_{\k^n}^i\{k_{i+1}-k_i\}
\end{align*}

\begin{Lemma}\label{lem:specializekernels}
\begin{enumerate}
\item Let $ \uk $ be such that $ k_i + k_{i+1} = m $ and let $ \uw $ be such that $ w_i = w_{i+1} $.   We have that
$$ \sC_1(\uk)^i_{\C^{n-1}}|_{Y(\uk)_\uw \times Y(d_i \cdot \uk)_{d_i \cdot \uw}} \cong \sC_1(\uk)^i_\uw. $$
A similar statement holds for $ \sC_2 $.
\item Let $ \uk $ be any sequence.  For any $\uw $, we have that
$$ \sT_1(\uk)^i_{\C^n}|_{Y(\uk)_\uw \times Y(s_i \cdot \uk)_{s_i \cdot \uw}} \cong \sT_1(\uk)^i_\uw. $$
Moreover, the kernel $ \sT_1(\uk)^i_{\C^n} $ is invertible. A similar statement holds for $ \sT_2, \sT_3, \sT_4 $.
\end{enumerate}
\end{Lemma}
\begin{Remark} In this Lemma, when $ \uw \ne \uo$, we mean an isomorphism of non-equivariant sheaves, whereas when $ \uw = \uo $, we mean an isomorphism of $ \C^\times$-equivariant sheaves.
\end{Remark}

\begin{proof}
The case of the cups and caps is clear. The invertibility of the kernel $ \sT_1(\uk)^i_{\C^n} $ follows from Theorem \ref{thm:kernel2}. The isomorphism also follows from Theorem \ref{thm:kernel2} using the fact that if $w_i \ne w_{i+1}$ then the restriction of the line bundle $\O_{Z^o(\uk)^i_{\C^n}}([Z_s^o(\uk)^i_{\C^n}])$ is trivial.
\end{proof}

\subsection{The definition of the coloured link homology} \label{sec:colourlink}

As before, let $K$ be a link with $r$ components $K_1, \dots, K_r$.  Choose a generic projection of the link.  As usual, consider horizontal slices of the link such that between any two horizontal slices, we see only one generator of the tangle category.  Fix a particular slice, and suppose it cuts $ n $ strands of $K$.  Let $ j_1, \dots, j_n  \in \{1, \dots, r \}$ denote the link components that the strands belong to (reading from left to right).  This data defines a map $ \C^r \rightarrow \C^n $ and we can form the base change $ Y(\uk)_{\C^r} $ of $ Y(\uk)_{\C^n} $ along this map.

Now consider two neighbouring slices with $ n $ and $ n'$ strands (necessarily we have $ n' \in \{n, n-2, n+2 \}$) labeled $ \uk, \uk' $ respectively.  We have maps $ \C^r \rightarrow \C^n $ and $ \C^r \rightarrow \C^{n'} $ and so we can form $ Y(\uk)_{\C^r} \times_{\C^r} Y(\uk')_{\C^r} $, by making the base change along the maps from $ \C^r $ followed by the fibre product.  To this pair of neighbouring slices, we will now associate a kernel in $ D^{\C^\times}(Y(\uk)_{\C^r} \times_{\C^r} Y(\uk')_{\C^r}) $.

The two slices differ by a crossing, a cap, or a cup.  The kernel we assign will be the one from the previous section pulled back to $ Y(\uk)_{\C^r} \times_{\C^r} Y(\uk')_{\C^r}$. For example, if it is a crossing of type 1 (this forces $ n' = n $ and $ \uk' = s_i \cdot \uk $) then we assign the kernel $ f^* \sT_1(\uk)^i_{\C^n}$, where
$$f: Y(\uk)_{\C^r} \times_{\C^r} Y(s_i \cdot \uk)_{\C^r} \rightarrow Y(\uk)_{\C^n} \times_{\C^n} Y(s_i \cdot \uk)_{\C^n}. $$

Now that we have associated a kernel to each pair of neighbouring slices, we compose them (relative to $ \C^r$) to obtain $\Psi(K)_{\C^r} \in D^{\C^\times}(\C^r)$. Finally, to obtain invariance under Reidemeister I, we shift this by $[\sum_k d_k k(m-k)]\{- \sum_k d_k k (m-k)\}$ where $d_k$ is the number of positive minus the number of negative crossings in $K$ involving two strands labeled $k$ ({\it c.f.} Remark \ref{rem:shift}).

\begin{Proposition}\label{prop:homology}
If $K$ is a link with $r$ components then $\Psi(K)_{\C^r} \in D^{\C^\times}(\C^r)$ is a link invariant. Moreover, for any $ \uw \in \k^r $, $\Psi(K)_{\C^r}|_{\uw}  \cong \Psi(K)_\uw$.
\end{Proposition}
\begin{Remark} As before, in this proposition, when $ \uw \ne \uo $, we mean an isomorphism in the derived category of vector spaces, whereas when $ \uw = \uo $, we mean an isomorphism in the derived category of graded vector spaces.
\end{Remark}
\begin{proof}
To check this is a link invariant we need to check the relations from Lemma \ref{lem:relations}. Let us prove the Reidemeister III move (the other relations are proved similarly). For simplicity of notation, let us assume we are dealing with type 1 crossings and we supress the labels $\uk$ from the notation.  Thus we would like to prove
\begin{equation}\label{eq:rel}
\sT^i_{\k^n}  \sT^{i+1}_{\k^n}  \sT^i_{\k^n}  (\sT^{i+1}_{\k^n})^{-1}  (\sT^{i}_{\k^n})^{-1} (\sT^{i+1}_{\k^n})^{-1} \cong \O_\Delta \in D^{\C^\times}(Y(\uk)_{\k^n} \times_{\k^n} Y(\uk)_{\k^n})
\end{equation}
where there are $n$ strands. We know this relation holds if we restrict to any $\uw \in \k^n$. More generally, since this relation only involves three strands (strands $i,i+1$ and $i+2$) it also holds if we restrict to any $\k^{i-1} \times (w_i,w_{i+1},w_{i+2}) \times \k^{n-i-2} \subset \k^n$. Since the deformation along $\w = (w,\dots,w)$ is trivial we may further assume that $w_i+w_{i+1}+w_{i+2} = 0$. Relation (\ref{eq:rel}) now follows from Lemma \ref{lem:families} where we take $B = \k^2 \cong \{(w_i,w_{i+1},w_{i+2}): w_i+w_{i+1}+w_{i+2}=0 \}$. This concludes the proof of Reidemeister move III.

All the other moves from Lemma \ref{lem:relations} are either clear (isotopy moves) or involve at most three strands (so the argument above applies).

The fact that $ \Psi(K)_\uw $ specializes correctly follows from Lemma \ref{lem:specializekernels}.
\end{proof}

\begin{Lemma}\label{lem:families}
Suppose $\pi: \tY \rightarrow B$ is a proper, flat family with $\dim(B) \le 2$ and $Pic(B)$ trivial. For $\sM, \sN \in D(\tY)$ suppose that $\sN$ is a sheaf which remains a sheaf when restricted to any fibre $ \tY_p$.  Assume also that $\End^0(\sN|_{\tY_p}) \cong \k$ and that $\sM|_{\tY_p} \cong \sN|_{\tY_p}$ for any $p \in B$. Then $\sM \cong \sN$.
\end{Lemma}
\begin{proof}
Since $\sM$ is a sheaf when restricted to any fibre $\tY_p$ it must be a sheaf on $\tY$. Now consider $\sA := \sHom(\sM,\sN) \in D(B)$ where the $\sHom$ is taken relative to the base (another way of saying this is $\pi_*(\sM^\vee \otimes \sN) \in D(B)$). By base change we know
$$\sA|_p \cong \Hom(\sM|_{\tY_p},\sN|_{\tY_p}) \cong \End^*(\sN|_{\tY_p})$$
which is a complex, supported in positive degrees, with the degree zero piece being one dimensional.

Now, using that $\dim(B) \le 2$, the spectral sequence computing $\sA|_p$ from $\sH^i(\sA)|_p$ shows that $\sH^0(\sA)|_p$ is supported in degree zero. This means that $\sH^0(\sA) = \sHom^0(\sM,\sN)$ is flat over $B$.

Any flat sheaf over a local ring is free which implies that a flat sheaf over any Noetherian scheme is locally free. This means that $\sH^0(\sA)$ is a vector bundle. Moreover, since $\dim H^0(\sA|_p) = 1$ this vector bundle is one dimensional. Finally, since $Pic(B)$ is trivial we must have $\sH^0(\sA) \cong \O_B$. This means that there is a section $f: \sM \rightarrow \sN$ which, when restricted to any point $p \in B$ induces an isomorphism. It follows that $\Cone(f)$ when restricted to any $\tY_p$ is zero. Thus $\Cone(f)=0$ and hence $f$ is an isomorphism.
\end{proof}

\subsection{Proof of Theorem \ref{thm:main2}}

Restricting $\Psi(K)_{\C^r} \in D^{\k^\times}(\k^r)$ to the line through $\underline{0} = (0,\dots,0)$ and $\uw = (w_1, \dots, w_r)$ gives us an object $\Psi(K)_{\C} \in D^{\k^\times}(\k)$. The fibre of $\Psi(K)_{\C}$ over zero is isomorphic to $\Psi(K)_{\underline{0}}$ whereas the general fibre is isomorphic to $\Psi(K)_\uw \cong \Psi(K_1)_{w_1} \otimes \dots \otimes \Psi(K_r)_{w_r}$.

Standard arguments now give us the required spectral sequence. More precisely, we can write $\Psi(K)_\C$ as a complex where the terms are free $\k[x]$-modules and the maps are $\k^\times$-equivariant. Then we can filter this complex with respect to the standard filtration on $\k[x]$ by degree. The associated graded is then isomorphic to $\Psi(K)_{\underline{0}}$ which gives us the spectral sequence.

\appendix

\section{Equivalences}\label{sec2}

In this section, to simplify notation, we will restrict to the case when we have two strands. In this case we have the variety
$$Z(k,l) := Z(k,l)^1_{(0,0)} = \{ L_0 \overset{k}{\underset{l}{\rightrightarrows}} \begin{matrix} L_1 \\ L_1' \end{matrix} \overset{l}{\underset{k}{\rightrightarrows}} L_2 : zL_2\subset L_1, zL_2 \subset L_1', zL_1 \subset L_0, zL_1' \subset L_0\} \subset Y(k,l) \times Y(l,k).$$
In \cite{CKL1}, we used the theory of geometric categorical $ \sl_2 $ actions to construct a natural equivalence $\T(k,l): D^{\C^\times}(Y(k,l)) \xrightarrow{\sim} D^{\C^\times}(Y(l,k))$ induced by a kernel $\sT(k,l)$. The purpose of this section is to describe this kernel more explicitly in terms of $Z(k,l)$. We follow the argument from \cite{C1} where this was done for $T^* \G(k,m)$ which is an open subset of $Y(k,m-k)$.

\subsection{The varieties $Z_s(k,l)$}

For simplicity we suppose $k \le l$. Following earlier notation we denote $N = \min(k+l,m)$. The variety $Z(k,l)$ contains $N-l+1$ components given by
$$Z_s(k,l) = \{(L_0,L_1,L_1',L_2) \in Z(k,l): \dim (\ker z|_{L_2/L_0}) \ge l+s \text{ and } \dim ((L_1 \cap L_1')/L_0) \ge k-s \}.$$
We define the open subscheme
$$Z^o(k,l) := \{(L_0,L_1,L_1',L_2) \in Z(k,l): \dim (\ker z|_{L_2/L_0}) + \dim ((L_1 \cap L_1')/L_0) \le k+l+1 \} \subset Z(k,l)$$
and $Z^o_s(k,l) := Z_s(k,l) \cap Z^o(k,l)$. Abusing notation slightly, we denote all the open inclusions $j: Z^o_s(k,l) \rightarrow Z_s(k,l)$.

\begin{Lemma}\label{lem:Z^o}
Each $Z_s(k,l)$ has a natural partial resolution
$$Z'_s(k,l) := \{ L_0 \xrightarrow{k-s} W_1 \overset{s}{\underset{l-k+s}{\rightrightarrows}} \begin{matrix} L_1 \\ L_1' \end{matrix} \overset{l-k}{\underset{k}{\rightrightarrows}} L_2 : zL_2 \subset W_1, zL_1 \subset L_0, zL_1' \subset L_0 \} \xrightarrow{\pi_s} Z_s(k,l)$$
where $\pi_s$ forgets $W_1$. $Z_s^o(k,l)$ is smooth and the restriction of $\pi_s$ to the preimage of $Z_s^o(k,l)$ is an isomorphism. The complements of $Z_s^o(k,l)$ in $Z'_s(k,l)$ and $Z_s(k,l)$ have codimension at least $3$ and $4$ repectively.
\end{Lemma}
\begin{proof}
$Z_s(k,l)$ has a natural resolution given by
$$Z''_s(k,l) := \{ L_0 \xrightarrow{k-s} W_1 \overset{s}{\underset{l-k+s}{\rightrightarrows}} \begin{matrix} L_1 \\ L_1' \end{matrix} \overset{l-k+s}{\underset{s}{\rightrightarrows}} W_2 \xrightarrow{k-s} L_2 : zL_2 \subset W_1, z W_2 \subset L_0\}.$$
Now $Z_s^o(k,l) \subset Z_s(k,l)$ is defined by the open condition
$$\dim (\ker z|_{L_2/L_0}) + \dim ((L_1 \cap L_1')/L_0) \le k+l+1.$$
Since $\dim (\ker z|_{L_2/L_0}) \ge l+s$ and $\dim ((L_1 \cap L_1')/L_0) \ge k-s$ there are two cases: either $\dim (\ker z|_{L_2/L_0}) = l+s$ or $\dim ((L_1 \cap L_1')/L_0) = k-s$. In the first case we recover $W_2$ as the kernel of $z$ and $W_1$ as the image of $z$ while in the second case we recover $W_1$ as $L_1 \cap L_1'$ and $W_2$ as $\spn(L_1,L_1')$. Thus $\pi$ is an isomorphism over $Z_s^o(k,l)$ (and in particular, this means $Z_s^o(k,l)$ is smooth).

Now, the complement $Z'_s(k,l) \setminus Z_s^o(k,l)$ is covered by three pieces:
\begin{itemize}
\item $\dim ((L_1 \cap L_1')/L_0) \ge k-s+2$ (where $s \ge 2$)
\item $\dim ((L_1 \cap L_1')/L_0) \ge k-s+1$ and $\dim (\ker z|_{L_2/L_0}) \ge l+s+1$ (where $s \ge 1$)
\item $\dim (z|_{L_2/L_0}) \ge l+s+2$ (where $s \le k-1$).
\end{itemize}
The dimension of the first piece can be computed as the dimension of its resolution
$$\{ L_0 \xrightarrow{k-s} W_1 \xrightarrow{2} W_1' \overset{s-2}{\underset{l-k+s-2}{\rightrightarrows}} \begin{matrix} L_1 \\ L_1' \end{matrix} \overset{l-k+s}{\underset{s}{\rightrightarrows}} W_2 \xrightarrow{k-s} L_2 : zL_2 \subset W_1, zW_2 \subset L_0 \}.$$
This variety is a sequence of iterated Grassmannian bundle and its dimension can be calculated to be $k(m-k)+l(m-l)-4$. The dimension of $Z_s(k,l)$ is $k(m-k)+l(m-l)$ so the codimension of the first piece is $4$.

The codimensions of the second and third pieces are computed similarly. For the second piece we use the resolution
$$\{ L_0 \xrightarrow{k-s} W_1 \xrightarrow{1} W_1' \overset{s-1}{\underset{l-k+s-1}{\rightrightarrows}} \begin{matrix} L_1 \\ L_1' \end{matrix} \overset{l-k+s+1}{\underset{s+1}{\rightrightarrows}} W_2 \xrightarrow{k-s-1} L_2 : zL_2 \subset W_1, zW_2 \subset L_0 \}$$
which has dimension $k(m-k)+l(m-l)-3$ (codimension $3$).

For the third piece we use the resolution
$$\{ L_0 \xrightarrow{k-s} W_1 \overset{s}{\underset{l-k+s}{\rightrightarrows}} \begin{matrix} L_1 \\ L_1' \end{matrix} \overset{l-k+s+2}{\underset{s+2}{\rightrightarrows}} W_2 \xrightarrow{k-s+2} L_2 : zL_2 \subset W_1, zW_2 \subset L_0 \}$$
which has dimension $k(m-k)+l(m-l)-4$ again.

To show that the codimension of $Z_s(k,l) \setminus Z_s^o(k,l) \subset Z_s(k,l)$ is at least four the same argument as above works except in the second case. There one needs to use the resolution
$$\{L_0 \xrightarrow{k-s+1} W_1 \overset{s-1}{\underset{l-k+s-1}{\rightrightarrows}} \begin{matrix} L_1 \\ L_1' \end{matrix} \overset{l-k+s+1}{\underset{s+1}{\rightrightarrows}} W_2 \xrightarrow{k-s-1} L_2 : zL_2 \subset W_1, zW_2 \subset L_0 \}$$
which again has dimension $k(m-k)+l(m-l)-4$.
\end{proof}

\begin{Corollary}\label{cor:Z^o}
$\pi_{s*} \O_{Z'_s(k,l)} \cong j_* \O_{Z^o_s(k,l)}$.
\end{Corollary}
\begin{proof}
This is an immediate consequence of Lemma \ref{lem:Z^o} (see \cite[Cor. 3.2]{C1}).
\end{proof}

\begin{Remark}\label{Rem1}
By Lemma \ref{lem:Z^o} we can identify $Z_s^o(k,l)$ with
\begin{eqnarray*}
\{ L_0 \xrightarrow{k-s} W_1 \overset{s}{\underset{l-k+s}{\rightrightarrows}} \begin{matrix} L_1 \\ L_1' \end{matrix} \overset{l-k+s}{\underset{s}{\rightrightarrows}} W_2 \xrightarrow{k-s} L_2 &:& zL_2 \subset W_1, zW_2 \subset L_0, \\
& &  \dim (\ker z|_{L_2/L_0}) + \dim ((L_1 \cap L_1')/L_0) \le k+l+1 \}
\end{eqnarray*}
where the isomorphism to $Z_s^o(k,l)$ is obtained by forgetting $W_1$ and $W_2$.
\end{Remark}

\begin{Lemma}\label{lem:intersections}
Using the identification of $Z_s^o(k,l)$ from Remark \ref{Rem1}, the intersection
\begin{align*}
D^o_{s,+} &:=
Z_s^o(k,l) \cap Z_{s+1}(k,l) \\
&= \{(L_\cdot,L'_\cdot) \in Z_s^o(k,l): \dim (\ker z|_{L_2/L_0}) = l+s+1 \text{ and } \dim ((L_1 \cap L_1')/L_0) = k-s \}
\end{align*}
is a divisor in $Z_s^o(k,l)$ which is the locus where $z: L_2/W_2 \rightarrow L_1/L_0 \{2\}$ is not an isomorphism. Thus
\begin{equation}\label{eq:A}
\O_{Z_s^o(k,l)}([D^o_{s,+}]) \cong \O_{Z_s^o(k,l)} \otimes \det(L_2/W_2)^\vee \otimes \det(W_1/L_0) \{2(k-s)\}.
\end{equation}
Similarly, the intersection
\begin{align*}
D^o_{s,-} &:=
Z_s^o(k,l) \cap Z_{s-1}(k,l) \\
&= \{(L_\cdot,L'_\cdot) \in Z_s^o(k,l): \dim (\ker z|_{L_2/L_0}) = l+s \text{ and } \dim ((L_1 \cap L_1')/L_0) = k-s+1 \}
\end{align*}
is a divisor in $Z_s^o(k,l)$ which is the locus where $L_1/W_1 \hookrightarrow W_2/L_1'$ is not an isomorphism. Thus
\begin{equation}\label{eq:B}
\O_{Z_s^o(k,l)}([D^o_{s,-}]) \cong \O_{Z_s^o(k,l)} \otimes \det(L_1/W_1)^\vee \otimes \det(W_2/L_1').
\end{equation}
Finally, $Z_s^o(k,l) \cap Z_{s'}(k,l) = \emptyset$ if $|s-s'| > 1$.
\end{Lemma}
\begin{proof}
The proof is precisely the same as that of \cite[Lem. 3.4]{C1} where $L_2$ plays the role of $\C^N$.
\end{proof}

\begin{Corollary}\label{cor:intersections}
We have
$$\O_{Z_s^o(k,l)}([D^o_{s,+}] - [D^o_{s,-}]) \cong \O_{Z_s^o(k,l)} \otimes \det(L_2/L_1)^\vee \otimes \det(L_1'/L_0) \{2(k-s)\} .$$
\end{Corollary}
\begin{proof}
This is a direct consequence of Lemma \ref{lem:intersections} since
$$\det(L_2/W_2)^\vee \otimes \det(W_1/L_0) \otimes \det(L_1/W_1) \otimes \det(W_2/L_1')^\vee \cong \det(L_2/L_1)^\vee \otimes \det(L_1'/L_0).$$
\end{proof}

\subsection{The kernel $\sT(k,l)$}

Following \cite{CKL1} we consider the complex $\Theta_{N-l} \rightarrow \Theta_{N-l-1} \rightarrow \dots \rightarrow \Theta_1 \rightarrow \Theta_0$ where $\Theta_s = \sF^{(l-k+s)} * \sE^{(s)}[-s]\{s\}$. Here
\begin{align*}
\sE^{(s)} &= \O_{W^{s}} \otimes \det(V/L_0)^{s} \otimes \det(L_1/V)^{-k+s} \{s(k-s)\} \\
\sF^{(l-k+s)} &= \O_{W^{l-k+s}} \otimes \det(L_2/L_1')^{-l+k-s} \otimes \det(L_1'/V)^{k} \{k(l-k+s)\}
\end{align*}
and
\begin{align*}
W^s &= \{L_0 \xrightarrow{k-s} V \xrightarrow{s} L_1 \xrightarrow{l} L_2, zL_2 \subset V \text{ and } zL_1 \subset L_0 \} \\
W^{l-k+s} &= \{L_0 \xrightarrow{k-s} V \xrightarrow{l-k+s} L_1' \xrightarrow{k} L_2, zL_2 \subset V \text{ and } zL_1' \subset L_0 \}.
\end{align*}
Notice that $\sE^{(s)} \in D^{\C^\times}(Y(k,l) \times Y(k-s,l+s))$ and $\sF^{(l-k+s)} \in D^{\C^\times}(Y(k-s,l+s) \times Y(l,k))$. The complex $\Theta_\bullet$ has a unique (right) convolution ${\rm Conv}(\Theta_\bullet)$. We denote
$$\sT(k,l) := {\rm Conv}(\Theta_\bullet)\{-k\} \in D^{\C^\times}(Y(k,l) \times Y(l,k)).$$
The main results of \cite{CKL2,CKL3} show that this kernel induces an equivalence $\T(k,l)$. We will now identify $\sT(k,l)$ more explicitly.

\begin{Remark}\label{rem:kernel}
In comparison to earlier work, the kernels we use here have all been conjugated by the line bundle $\det(L_2/L_0)^{l} \otimes \det(L_1/L_0)^k$ on $Y(k,l)$. Since we try to emphasize the similarities with the results in \cite{C1} we will subsequently see the line bundle $\rho := \det(L_2/L_1)^k \otimes \det(L_2/L_1')^{-l}$ on $Y(k,l) \times Y(l,k)$ show up.
\end{Remark}

\begin{Proposition}\label{prop:kernel}
We have
$$\Theta_s \cong j_* \O_{Z_s^o(k,l)} \otimes \det(L_2/L_1)^{-s} \otimes \det(L_1'/L_0)^s \otimes \rho [-s] \{kl-(k-s)^2 + s\}$$
as a sheaf on $Y(k,l) \times Y(l,k)$.
\end{Proposition}
\begin{proof}
Ignoring the grading shift we have
$$\Theta_s \cong \pi_{13*}(\pi_{12}^* \O_{W^s} \otimes \pi_{23}^* \O_{W^{l-k+s}} \otimes \det(V/L_0)^{s} \otimes \det(L_1/V)^{-k+s} \otimes \det(L_2/L_1')^{-l+k-s} \otimes \det(L_1'/V)^k)$$
where $\pi_{13}$ forgets $V$. One can check as in the proof of \cite[Prop. 3.7]{C1} that $\pi_{12}^{-1}(W^s)$ and $\pi_{23}^{-1}(W^{l-k+s})$ intersect in the expected dimension to give $Z_s'(k,l)$. Thus we get
$$\Theta_s \cong \pi_{13*}(\O_{Z_s'(k,l)} \otimes \det(L_2/L_1)^{-s} \otimes \det(L_1'/L_0)^{s} \otimes \rho) \cong j_* \O_{Z_s^o(k,l)} \otimes \det(L_2/L_1)^{-s} \otimes \det(L_1'/L_0)^{s} \otimes \rho.$$
Finally, the grading shift is equal to $s(k-s)+k(l-k+s)+s = kl - (k-s)^2 + s$.
\end{proof}

\begin{Proposition}\label{prop:equivalence}
We have $\sT(k,l) \cong j_* \sL(k,l)$ where $\sL(k,l)$ is the line bundle on $Z^o(k,l)$ uniquely determined by the restrictions
$$\sL(k,l)|_{Z^o_s(k,l)} \cong \O_{Z^o_s(k,l)}([D^o_{s,+}]) \otimes \det(L_2/L_1)^{-s} \otimes \det(L_1'/L_0)^{s} \otimes \rho \{kl - (k-s)(k-s-1)\}$$
and $j$ is the open inclusion $Z^o(k,l) \xrightarrow{j} Z(k,l)$.
\end{Proposition}
\begin{proof}
This follows by precisely the same argument used to proof Theorem \cite[Thm. 3.8]{C1}.
\end{proof}

This proposition shows that $ \sT(k,l) $ is isomorphic to the kernel $ \sT_1(k,l)^1_\uo $ defined in the main part of the paper.

\subsection{The deformed kernel $\sT(k,l)_{\C}$}

Recall the deformation
\begin{align*}
Z(k,l)_{\C^2} := Z(k,l)_{\C^2}^1 = & \{ 0 \overset{k}{\underset{l}{\rightrightarrows}} \begin{matrix} L_1 \\ L_1' \end{matrix} \overset{l}{\underset{k}{\rightrightarrows}} L_2: (z-w_1I)|_{L_2/L_1}=0, (z-w_2I)|_{L_2/L_1'}=0 \\
& (z-w_2I)|_{L_1/L_0}=0, (z-w_1I)|_{L_1'/L_0}=0 \}.
\end{align*}
We denote by $Z(k,l)_{\C}$ its restriction to $(w,-w)$ (we restrict to this locus for convenience and because the deformation in the direction $w_1=w_2$ is trivial). Inside $Z(k,l)_\C$ we have the open subscheme $Z^o(k,l)_{\C}$ defined as the locus where
$$k+l+1 \ge \dim (\ker(z-wI)|_{L_2/L_0}) + \dim ((L_1 \cap L_1')/L_0).$$
Notice that $Z^o(k,l)_{\C}$ contains all the fibres of $Z(k,l)_{\C}$ over any $w \ne 0$ since in this case we have $L_1 \cap L_1' = L_0$.

\begin{Remark}
The scheme $Z(k,l)_{\C}$ is now irreducible. The subschemes $Z_s(k,l)$ are divisors though not necessarily Cartier. However, things are much nicer inside $Z^o(k,l)_{\C}$ where each $Z^o_s(k,l) \subset Z^o(k,l)_{\C}$ is Cartier. We denote by $j$ the inclusions $j: Z^o(k,l)_{\C} \rightarrow Z(k,l)_{\C}$.
\end{Remark}

\begin{Theorem}\label{thm:kernel2}
The kernel $\sT(k,l)_{\C} := j_* \sL(k,l)_{\C} \in D^{\C^\times}(Y(k,l)_{\C} \times_{\C} Y(l,k)_{\C})$ is invertible, where
$$\sL(k,l)_{\C} := \O_{Z^o(k,l)_{\C}}(\sum_{s=0}^{N-l} \binom{s+1}{2} [Z^o_s(k,l)]) \otimes \rho \{k(l-k-1)\}.$$
Moreover, $\sT(k,l)_{\C}$ restricts to $\sT(k,l)$ over $Y(k,l) \times Y(l,k)$, which is the fibre over $w=0$.
\end{Theorem}
\begin{proof}
This follows in exactly the same way as \cite[Thm. 4.1]{C1}.
\end{proof}

\subsection{Inverses}

We will now identify explicitly the left adjoints $\sT(k,l)^L$ and $\sT(k,l)_{\C}^L$.

\begin{Theorem}\label{thm:inverse}
We have $\sT(k,l)_{\C}^L \cong \sigma \bigl(j_* \sL'(k,l)_{\C} \bigr) \in D^{\C^\times}(Y(l,k)_{\C} \times_{\C} Y(k,l)_{\C})$ where
$$\sL'(k,l)_{\C} := \O_{Z^o(k,l)_{\C}}(\sum_{s=0}^{N-l} \binom{s}{2} [Z_s^o(k,l)]) \otimes \det(L_1/L_0)^l \otimes \det(L_1'/L_0)^{-k} \{k(l-k+1)\}.$$
\end{Theorem}
\begin{Remark}
The kernel $\sT(k,l)^L$ is induced by the restriction of $\sT(k,l)_{\C}^L$ to $Y(k,l) \times Y(l,k)$.
\end{Remark}
\begin{proof}
We will use that
$$\sT(k,l)_{\C}^L = \sigma(j_* (\omega_{Z^o(k,l)_{\C}} \otimes \sL(k,l)_{\C}^\vee \otimes \pi_1^* \omega^\vee_{Y(k,l)_{\C}}) \otimes \rho^{-1}).$$
The proof will follow that of \cite[Thm. 5.3]{C1} by showing that $\sigma(\sT(k,l)_{\C}^L)$ and $j_* \sL'(k,l)_{\C}$ restrict to the same thing on every fibre.

On a fibre over $w \ne 0$ this is easy to see. Now consider the fibre over $w=0$. By Proposition \ref{prop:dualizingsheaf} we have that $\omega_{Z^o(k,l)_{\C}} \otimes \sL(k,l)_{\C}^\vee$ is isomorphic to
\begin{align*}
& \O_{Z^o(k,l)_{\C}}(\sum_{s=0}^{N-l} s^2 [Z_s^o(k,l)]) \otimes L^{(l-k)} \otimes \det(L_2/L_0)^m \{a\} \otimes \O_{Z^o(k,l)_{\C}}(\sum_{s=0}^{N-l} - \binom{s+1}{2} [Z_s^o(k,l)]) \{b\} \\
\cong& \O_{Z^o(k,l)_{\C}}(\sum_{s=0}^{N-l} \binom{s}{2} [Z_s^o(k,l)]) \otimes L^{(l-k)} \otimes \det(L_2/L_0)^m \{-2m(k+l)-k(k+l-1)-2\}
\end{align*}
where $L = \det(L_2/L_1)^{\vee} \otimes \det(L_1'/L_0)$, $a = -2k^2-2m(k+l)-2$ and $b = -k(l-k-1)$. Now
$$\omega_{Y(k,l)_{\C}}|_{Y(k,l)} \cong \omega_{Y(k,l)} \otimes \O_{Y(k,l)}(-[Y(k,l)]) \cong \O_{Y(k,l)} \otimes \det(L_2/L_0)^m \{-2m(k+l)-2kl-2\}$$
where we use that $\O_{Y(k,l)_\C}([Y(k,l)]) \cong \O_{Y(k,l)_\C} \{2\}$ and that $\omega_{Y(k,l)} \cong \det(L_2/L_0)^m \{-2m(k+l)-2kl\}$ (see \cite[Lem. 5.7]{CKL1}). Moreover, $L^{l-k} \otimes \rho^{-1} \cong \det(L_1/L_0)^l \otimes \det(L_1'/L_0)^{-k}$. Hence
\begin{align*}
&\sigma(\sT(k,l)_\C^L)|_{Y(k,l) \times Y(l,k)} \\
\cong& j_* (\O_{Z^o(k,l)_\C} (\sum_{s=0}^{N-l} \binom{s}{2} [Z_s^o(k,l)])|_{Y(k,l) \times Y(l,k)} \otimes \det(L_1/L_0)^l \otimes \det(L_1'/L_0)^{-k} \{k(l-k+1)\} \\
\cong& j_* \sL_\C'(k,l)|_{Y(k,l) \otimes Y(l,k)}.
\end{align*}
This is what we need to prove.
\end{proof}

\begin{Corollary}\label{cor:adjoint}
We have
$$\sT(k,l)_\C^L \cong \sigma \Bigl( \sT(k,l)_\C \otimes \det(L_1/L_0)^{k+l-1} \otimes \det (L_1'/L_0)^{-l-k-1} \otimes \det (L_2/L_0)^{l-k+1} \Bigr).$$
\end{Corollary}
\begin{proof}
Using Lemma \ref{lem:L} we have
$$\sT(k,l)_\C \cong j_* \O_{Z^o(k,l)_\C}(\sum_{s=0}^{N-l} \binom{s}{2} [Z^o_s(k,l)]) \otimes \det(L_2/L_1)^\vee \otimes \det(L_1'/L_0) \otimes \rho \{k(l-k-1)+2k\}.$$
On the other hand,
$$\sigma(\sT(k,l)_\C^L) \cong j_* \O_{Z^o(k,l)_\C}(\sum_{s=0}^{N-l} \binom{s}{2} [Z^o_s(k,l)]) \otimes \det(L_1/L_0)^l \otimes \det(L_1'/L_0)^{-k} \{k(l-k+1)\}.$$
The result follows.
\end{proof}

Finally, we have the following identity. Although we do not use it in this paper it will be useful elsewhere in order to prove that the affine braid group acts. 

\begin{Corollary}\label{cor:affinebraid}
We have $\sT(l,k) * \Delta_* \det(L_1/L_0)^\vee * \sT(k,l) \cong \Delta_* \det(L_2/L_1)^\vee$.
\end{Corollary}
\begin{proof}
The kernel $\sT(l,k) \in D^{\C^\times}(Y(l,k) \times Y(k,l))$ is given by ${\rm Conv}(\Theta'_{N-l} \rightarrow \dots \rightarrow \Theta'_1 \rightarrow \Theta'_0) \{-k\}$ where $\Theta'_s = \sF^{(s)} * \sE^{(l-k+s)} [-s]\{s\}$. Here
\begin{align*}
\sE^{(l-k+s)} &= \O_{W^{l-k+s}} \otimes \det(V/L_0)^{l-k+s} \otimes \det(L_1'/V)^{-k+s} \{(k-s)(l-k+s)\} \\
\sF^{(s)} &= \O_{W^{s}} \otimes \det(L_2/L_1)^{-s} \otimes \det(L_1/V)^{l} \{sl\}
\end{align*}
with $W^s$ and $W^{l-k+s}$ the same varieties as before. A similar calculation as that in Proposition \ref{prop:kernel} shows that
$$\Theta'_s \cong i_* j_* \O_{Z_s^o(k,l)} \otimes \det(L_2/L_1)^{-s} \otimes \det(L_1'/L_0)^s \otimes \rho' [-s] \{kl - (k-s)^2 + s\}$$
where $\rho' = \det(L_1/L_0)^l \otimes \det(L_1'/L_0)^{-k}$. It follows that, as objects in $D^{\C^\times}(Y(k,l) \times Y(l,k))$, we have
$$\sigma(\Theta_s') \cong \Theta_s \otimes \rho^{-1} \otimes \rho' \cong \Theta_s \otimes \det(L_1/L_0)^{k+l} \otimes \det(L_1'/L_0)^{-k-l} \otimes \det(L_2/L_0)^{l-k}.$$
Thus we obtain
$$\sigma(\sT(l,k)) \cong \sT(k,l) \otimes \det(L_1/L_0)^{k+l} \otimes \det(L_1'/L_0)^{-k-l} \otimes \det(L_2/L_0)^{l-k}.$$
Putting this together with Corollary \ref{cor:adjoint} we obtain
$$\sT(l,k) \cong \sT(k,l)^L \otimes \det(L_1/L_0) \otimes (L_1'/L_0) \otimes \det(L_2/L_0)^{-1}$$
from which the result follows.
\end{proof}

\subsection{Some technical calculations}

The rest of this section contains some technical computations that we used in earlier sections.

\begin{Lemma}\label{lem:L}
We have
\begin{equation}\label{eq:5}
\O_{Z^o(k,l)_\C} \otimes \det(L_2/L_1)^\vee \otimes \det(L_1'/L_0) \cong \O_{Z^o(k,l)_\C}(\sum_{s=1}^{N-l} s [Z^o_s(k,l)]) \{-2k\}.
\end{equation}
\end{Lemma}
\begin{proof}
The natural inclusion map $L_1'/L_0 \rightarrow L_2/L_1$ on $Z^o(k,l)_\C$ is an isomorphism over any fibre over $w \ne 0$ because $L_1 \cap L_1' = L_0$ in this case. Hence (\ref{eq:5}) holds when restricted to a general fibre.

Since every line bundle on the base $\C$ is trivial it remains to show that
$$\O_{Z^o(k,l)} \otimes \det(L_2/L_1)^\vee \otimes \det(L_1'/L_0) \cong \O_{Z^o(k,l)_\C}(\sum_{s=1}^{N-l} s [Z^o_s(k,l)]) \{-2k\}|_{Z^o(k,l)}.$$
It suffices to show that for every $t$ we have
$$\O_{Z^o_t(k,l)} \otimes \det(L_2/L_1)^\vee \otimes \det(L_1'/L_0) \cong \O_{Z^o(k,l)_\C}(\sum_{s=1}^{N-l} s [Z^o_s(k,l)])|_{Z^o_t(k,l)} \{-2k\}.$$
Now we have
\begin{align*}
& \O_{Z^o(k,l)_\C}(\sum_{s=1}^{N-l} s [Z^o_s(k,l)])|_{Z^o_t(k,l)} \{-2k\} \\
\cong& \O_{Z^o(k,l)_\C}((t-1)[Z^o_{t-1}(k,l)] + t[Z^o_t(k,l)] + (t+1)[Z^o_{t+1}(k,l)])|_{Z^o_t(k,l)} \{-2k\} \\
\cong& \O_{Z^o(k,l)_\C}(-[Z^o_{t-1}(k,l)] + [Z^o_{t+1}(k,l)])|_{Z^o_t(k,l)} \{2t-2k\} \\
\cong& \O_{Z^o_t(k,l)}(-[D^o_{t,-}] + [D^o_{t,+}]) \{2(t-k)\} \\
\cong& \O_{Z^o_t(k,l)} \otimes \det(L_2/L_1)^\vee \otimes \det(L_1'/L_0)
\end{align*}
where the first isomorphism uses that $Z^o_t(k,l) \cap Z^o_{t'}(k,l) = \emptyset$ if $|t-t'| > 1$, the second uses that $\O_{Z^o(k,l)_\C}(\sum_{s=0}^{N-l}[Z^o_s(k,l)]) \cong \O_{Z^o(k,l)_\C} \{2\}$, while the last uses Corollary \ref{cor:intersections}.
\end{proof}

\begin{Proposition}\label{prop:dualizingsheaf}
The dualizing sheaf of $Z^o(k,l)_\C$ is
\begin{equation}\label{eq:3}
\omega_{Z^o(k,l)_\C} \cong \O_{Z^o(k,l)_\C}(\sum_{s=0}^{N-l} s^2 [Z_s^o(k,l)]) \otimes L^{(l-k)} \otimes \det(L_2/L_0)^{m} \{-2k^2 - 2m(k+l) - 2\}
\end{equation}
where $L = \det(L_2/L_1)^{\vee} \otimes \det(L_1'/L_0)$.
\end{Proposition}
\begin{proof}
Away from the central fibre $Z(k,l)_\C$ is isomorphic to $Y(k) \times Y(l)$. Over such a fibre we have $L_2/L_1 \cong L_1'/L_0$. So, ignoring the $\C^\times$-equivariance, the right hand side of (\ref{eq:3}) restricts on this general fibre to give $\det(L_2/L_0)^m$ which is indeed the canonical bundle. It remains to show that $\omega_{Z^o(k,l)}$ is isomorphic to
$$(\O_{Z^o(k,l)_\C}(\sum_{s=0}^{N-l} s^2 [Z_s^o(k,l)] + [Z^o(k,l)]) \otimes L^{(l-k)} \otimes \det(L_2/L_0)^m \{-2k^2 - 2m(k+l) - 2\})|_{Z^o(k,l)}.$$
Now $\O_{Z^o(k,l)_\C}([Z^o(k,l)]) \cong \O_{Z^o(k,l)_\C} \{2\}$. Also, $Z^o(k,l)$ is the union of $N-l+1$ smooth components $Z^o_t(k,l)$ where $Z^o_t(k,l) \cap Z^o_{t'}(k,l) = \emptyset$ if $|t-t'| > 1$. Thus
$$\omega_{Z^o(k,l)}|_{Z^o_t(k,l)} \cong \omega_{Z^o_t(k,l)}([D^o_{t,-}] + [D^o_{t,+}])$$
so it suffices to show that
\begin{equation}\label{eq:1}
\omega_{Z^o_t(k,l)}([D^o_{t,-}] + [D^o_{t,+}]) \cong (\O_{Z^o(k,l)_\C}(\sum_{s=0}^{N-l} s^2 [Z_s^o(k,l)]) \otimes L^{(l-k)} \otimes \det(L_2/L_0)^m \{-2k^2-2m(k+l)\})|_{Z^o_t(k,l)}
\end{equation}
for every $t=0, \dots, N-l$. Now $\O_{Z^o(k,l)_\C}([Z_t^o(k,l)]) \cong \O_{Z^o(k,l)_\C}(\sum_{s \ne t} - [Z_s^o(k,l)]) \{2\}$ so the right hand side of (\ref{eq:1}) equals
\begin{align*}
& \O_{Z^o_t(k,l)}((-2t+1)[D^o_{t,-}] + (2t+1)[D^o_{t,+}]) \otimes L^{(l-k)} \otimes \det(L_2/L_0)^m \{-2k^2 + 2t^2 - 2m(k+l)\} \\
\cong& \O_{Z^o_t(k,l)}([D^o_{t,-}] + [D^o_{t,+}]) \otimes L^{(l-k+2t)} \otimes \det(L_2/L_0)^m  \{4t(k-t) - 2k^2 + 2t^2 - 2m(k+l)\}
\end{align*}
where we used Corollary \ref{cor:intersections} to get the last isomorphism. By Lemma \ref{lem:dualizingsheaf} this equals the left hand side of (\ref{eq:1}).
\end{proof}

\begin{Lemma}\label{lem:dualizingsheaf}
The dualizing sheaf of $Z^o_s(k,l)$ is
$$\omega_{Z^o_s(k,l)} \cong (\det(L_2/L_1)^{\vee} \otimes \det(L_1'))^{(l-k+2s)} \otimes \det(L_2/L_0)^m \{-2m(k+l)-2(k-s)^2\}.$$
\end{Lemma}
\begin{proof}
As in the proof of Lemma \ref{lem:Z^o}, $Z_s(k,l)$ has a natural resolution
$$Z''_s(k,l) := \{ L_0 \xrightarrow{k-s} W_1 \overset{s}{\underset{l-k+s}{\rightrightarrows}} \begin{matrix} L_1 \\ L_1' \end{matrix} \overset{l-k+s}{\underset{s}{\rightrightarrows}} W_2 \xrightarrow{k-s} L_2 : zL_2 \subset W_1, zW_2 \subset L_0\}$$
where the map $\pi: Z''_s(k,l) \rightarrow Z_s(k,l)$ is an isomorphism over $Z_s^o(k,l)$. Hence $\omega_{Z^o_s(k,l)} \cong \omega_{Z''_s(k,l)}|_{Z^o_s(k,l)}$.

Now consider the projection map
\begin{equation}\label{eq:2}
p: Z''_s(k,l) \rightarrow \{L_0 \xrightarrow{k-s} W_1 \xrightarrow{l-k+2s} W_2 \xrightarrow{k-s} L_2:  zL_2 \subset W_1, zW_2 \subset L_0\}
\end{equation}
which forgets $L_1$ and $L_1'$. This is a $\G(s, W_2/W_1) \times \G(l-k+s,W_2/W_1)$ fibration so the relative cotangent bundle is
\begin{eqnarray*}
\omega_{p}
&\cong& (\det(L_1/W_1)^{l-k+s} \otimes \det(W_2/L_1)^{-s}) \otimes (\det(L_1'/W_1)^{s} \otimes \det(W_2/L_1')^{-l+k-s}) \\
&\cong& (\det(L_2/L_1)^\vee \otimes \det(L_1'))^{l-k+2s} \otimes (\det(L_2/W_2)^\vee \otimes \det(W_1))^{-l+k-2s}.
\end{eqnarray*}
On the other hand, the canonical bundle of the right hand side of (\ref{eq:2}) is
$$\det(L_2/L_0)^m \otimes \det(L_2/W_2)^{-l+k-2s} \otimes \det(W_1/L_0)^{l-k+2s} \{-2m(k+l)-2(k-s)^2\}$$
(see \cite[Lem. 5.7]{CKL1}). Putting this together we get the result.
\end{proof}

\end{document}